\documentclass[]{elsarticle}
\usepackage[T1]{fontenc}
\usepackage[latin9]{inputenc}
\usepackage{color}
\usepackage{amsmath}
\usepackage{amssymb}
\usepackage{amsthm}
\usepackage{graphicx}
\usepackage{esint}
\usepackage{url}
\usepackage{graphicx}
\usepackage{float}

\usepackage{color}
\newcommand{\blue}{\color{black}}
\newcommand{\black}{\color{black}}

\newcommand{\e}{\mathrm{e}}

\newcommand{\pO}{\vert_{\partial\Omega}}

\newtheorem{theorem}{Theorem}

\begin{document}

\title{A comparison of boundary correction methods for Strang splitting}
\journal{DCDS-B}

\author[uibk]{Lukas Einkemmer}
\author[uibk]{Alexander Ostermann\corref{cor1}}\ead{alexander.ostermann@uibk.ac.at}
\address[uibk]{Department of Mathematics, University of Innsbruck, Austria}
\cortext[cor1]{Corresponding author}

\begin{abstract}
In this paper we consider splitting methods in the presence of non-homogeneous boundary conditions. In particular, we consider the corrections that have been described and analyzed in Einkemmer, Ostermann 2015 and Alonso-Mallo, Cano, Reguera 2016. The latter method is extended to the non-linear case, and a rigorous convergence analysis is provided. We perform numerical simulations for diffusion-reaction, advection-reaction, and dispersion-reaction equations in order to evaluate the relative performance of these two corrections. Furthermore, we introduce an extension of both methods to obtain order three locally and evaluate under what circumstances this is beneficial.
\end{abstract}


\maketitle

\section{Introduction}
The present paper is concerned with the numerical solution of the partial differential equations that can be written as an abstract evolution equation
\begin{equation} \label{eq:problem}
\partial_t u = Au + f(u), \qquad u\pO=b, \qquad u(0)=u_0,
\end{equation}
where $f(u)$ is a non-stiff reaction (usually $f$ does only depend on $u$ but not on its derivatives) and $A$ is a linear differential operator. The latter is the reason for the stiffness of the spatial semi-discretization. In the present work we will consider \blue in particular \black $A=\partial_{xx}$ (a diffusion-reaction problem), $A=\partial_x(a(x){}\cdot{})$ (an advection-reaction problem), and $A=i\partial_{xx}$ (a dispersion-reaction problem). \blue Let us emphasize that dependent on the problem type, boundary conditions can only be prescribed at a certain part of the boundary. For instance, the before mentioned advection-reaction problems only admit boundary conditions at the inflow boundary. On the other hand, for parabolic problems we have to prescribe appropriate conditions on the whole boundary. In the part of the analysis common to all problem classes introduced above, we will use $\partial \Omega$ to denote the part of the boundary at which conditions are imposed. This is mainly done for notational simplicity. \black

A popular method to solve this class of partial differential equations is splitting. The basic idea of splitting is to decompose equation (\ref{eq:problem}) into
\begin{equation} \partial_t v = Av, \qquad v\pO=b, \qquad v(0)=v_0 \label{eq:pf-A} \end{equation}
and
\begin{equation} \partial_t w = f(w), \qquad w(0)=w_0. \label{eq:pf-f} \end{equation}
In order to simplify the notation we will use
\[ v(t) = \varphi_{t}^A(v_0) \quad \text{and} \quad w(t)=\varphi_{t}^f(w_0) \]
to denote these two partial flows. The Strang splitting procedure can then be written as
\[ u_{n+1} = S_{\tau}(u_n) = \varphi_{\frac{\tau}{2}}^f \circ \varphi_{\tau}^A \circ \varphi_{\frac{\tau}{2}}^f (u_n). \]
In the absence of boundary conditions, it is second order accurate. Clearly, splitting is only viable if a procedure exists to efficiently solve the two partial flows (\ref{eq:pf-A}) and (\ref{eq:pf-f}). However, since the reaction is not stiff and good preconditioners are known for a large class of linear operators $A$, such an approach can be significantly more efficient than applying a monolithic implicit Runge--Kutta or multistep method (which requires both a nonlinear and a linear solver).
Consequently, a significant body of research has been devoted to splitting methods. Among their applications are diffusion-reaction equations \cite{descombes2001,gerisch2002,einkemmer2015,hundsdorfer2003book,einkemmer2016}, advection-reaction equations \cite{hundsdorfer1995,spee1998,hundsdorfer2003book}, diffusion-reaction-advection equations \cite{dawson1991,hundsdorfer2003book}, Schr\"odinger-type equations \cite{bao2002,lubich2008,faou2012}, dispersive equations \cite{klein2011,holden2013,einkemmer1408}, and kinetic equations \cite{cheng1976,grandgirard2006,einkemmer1207,einkemmer1401,casas2015,einkemmer1602}.

In the present paper, however, we are concerned with splitting in the presence of non-homogenous Dirichlet boundary conditions. In this case the Strang splitting procedure is only first order accurate \cite{hundsdorfer1995,hundsdorfer2003book,einkemmer2015,alonsomallo2016} which is a significant problem in applications.
The classic approach to address this problem can be found in \cite{leveque1983,connors2014,alonsomallo2015,alonsomallo2016}. However, more recently an alternative approach has been introduced and analyzed in \cite{einkemmer2015,einkemmer2016}. In the following we will provide a brief description of both methods. A more detailed explanation is given in section~\ref{sec:corrections}.

The classic approach is based on a modified time-dependent boundary condition for equation (\ref{eq:pf-A}). This method has been analyzed in \cite{alonsomallo2016}, where the authors consider the application to a linear partial differential equation using a dimension splitting approach. This approach can also be seen in the context of similar corrections that have been developed for implicit Runge--Kutta \cite{carpenter1993} and Lawson-type methods \cite{alonsomallo2015}. In the following we will refer to this method by the abbreviation TDBC (time-dependent boundary correction).

In \cite{einkemmer2015} a correction based on a different partitioning of the two partial flows has been introduced. For diffusion-reaction equations it has been shown that this numerical method can attain second order accuracy in the presence of non-trivial boundary conditions. Since the scheme is based on enforcing a so-called compatibility condition (between the boundary data and the reaction term), we will use the abbreviation CEC (compatibility enforcing correction) to refer to this approach in the remainder of the paper.

\blue We note that order reduction in splitting methods can also be caused by non-smooth data. This, however, will not be studied here. Assuming that all occurring functions are sufficiently smooth, we restrict our attention to the effect of Dirichlet boundary conditions on the order of the method, and on strategies to remedy this situation.\black

In the present paper we \blue discuss an extension of \black the TDBC approach to \blue non-linear problems\black. This is rather straightforward and is the subject of section \ref{sec:corrections}. In addition, we will show how to extend both correction methods (TDBC and CEC) to obtain a local error of order three.
Then we will perform a convergence analysis for the TDBC correction in the non-linear case (section \ref{sec:convergence}). Note that this analysis is based on a similar analysis for the CEC method that has been performed in \cite{einkemmer2015}.
In section \ref{sec:numerics} we compare these two methods for three different classes of partial differential equations (diffusion-reaction, advection-reaction, and dispersion-reaction) and evaluate under which circumstances it is beneficial to use the third order correction. Finally, we conclude in section \ref{sec:conclusion}.

\section{Splitting corrections \label{sec:corrections}}

In this section we describe the two methods used to attain higher order splitting schemes in the presence of non-trivial boundary conditions in some detail.

However, before doing so, we will state the so-called compatibility conditions for the problem under consideration. These are derived by taking time derivatives of equation (\ref{eq:problem}) and using the fact that, for time-invariant boundary data, $\partial_t u \pO = \partial_t b = 0$. From this procedure we obtain the first compatibility condition
\[ Au + f(u) \pO = 0 \]
and the second compatibility condition
\[ A^2 u + Af(u) \pO = 0, \]
both of which we will use extensively in the remainder of the paper.

For notational simplicity, we will mainly discuss time-invariant boundary data in this paper. We note, however, that our analysis easily generalizes to the time-dependent case with the obvious modifications. For instance, the right-hand sides of the first and the second compatibility conditions have to be replaced by $\partial_t b$ and \blue $\partial_{tt}b - f'(b)\partial_t b$, \black respectively.

First, let us consider the CEC method (introduced in \cite{einkemmer2015}). This method starts with the observation that no order reduction is observed for homogeneous Dirichlet boundary conditions and the (physically reasonable) assumption that $f(0)=0$. This observation can be easily extended to more general reactions. In this case the requirement is that reaction leaves the boundary data invariant (i.e.~$f(b)=b$).

We then introduce a correction $q$ (which does \textit{not} depend on $u$ or $t$ for time-invariant Dirichlet boundary conditions) and instead of equations (\ref{eq:pf-A}) and~(\ref{eq:pf-f}) we solve
\[ \qquad \partial_t v = Av + q, \qquad v\pO=b, \qquad v(0)=v_0 \]
and
\[ \partial_t w = f(w) - q, \qquad w(0)=w_0. \]
The correction is chosen such that
\[ q\vert_{\partial\Omega}=f(b) \]
which leaves considerable freedom in how $q$ is determined inside the domain. Note, however, that the numerical simulations conducted in \cite{einkemmer2016} suggest that smooth functions with not too large derivatives perform best. It has been rigorously shown in \cite{einkemmer2015} that this correction gives a local error of order two and (using the parabolic smoothing property \blue in the framework of analytic semigroups\black) also a global error of order two.

In the previous work \cite{einkemmer2015,einkemmer2016} this correction has been used to obtain a second order accurate scheme for parabolic problems. It is, however, possible to extend this method in order to obtain a local error of order three. From the convergence analysis performed in \cite{einkemmer2015} it follows that we have to ensure that $A(f(u)-q)\vert_{\partial\Omega} = 0$. That is, we have to find a correction $q$ such that
\[ q\vert_{\partial\Omega} = f(b), \qquad Aq\vert_{\partial\Omega} = Af(u)\vert_{\partial\Omega}. \]
This amounts to the construction of a correction that satisfies two boundary conditions (one for $q$ and one for $Aq$). Since $q$ does not have to satisfy any condition in the interior of the domain this is certainly possible in principle. However, evaluating $Af(u)$ by numerical differentiation \blue will introduce additional contributions the local error. This might  increase the error constant of the corrected Strang splitting \black significantly.

As an example \blue where $Af(u)$ can easily be evaluated at the boundary, let us consider the first order differential operator of the form $A=\nabla\cdot(a(x){}\cdot{})$ with inflow boundary $\Gamma$\black. There we have
\[ Af(u) = \nabla\cdot(af(u)) = f^\prime(u) Au + (\nabla \cdot a)(f(u)-f^\prime(u)u) \]
and thus on the boundary
\blue \[ Af(u) \vert_{\Gamma} = - f^\prime(b)f(b) + (\nabla\cdot a)(f(b)-f^\prime(b)b). \]
In the previous equation we have used the first compatibility condition (i.e.{} $Au\vert_{\Gamma}=-f(u)\vert_{\Gamma}$). \black
Thus, we can completely determine the correction by using the boundary data and function evaluations of $f$.

Unfortunately, this is not true for more general operators. For example, for $A=c\Delta$ we have
\[ Af(u) = c f^{\prime\prime}(u) (\nabla u)^2 + f^\prime(u) A u \]
which on the boundary gives
\[ Af(u)\vert_{\partial\Omega} = c f^{\prime\prime}(b) (\nabla u\pO)^2 - f^\prime(b) f(b). \]
Note that this still requires to evaluate the first derivative of $u$ (an improvement over a direct application of $A$ which would require two derivatives). However, we are not able to completely eliminate the need for numerical differentiation.

The advantage of the CEC approach is that the modification of the partial flows are mild. \blue For time-invariant $b$\black, the correction only adds an inhomogeneity that is independent of time and independent of the solution. In addition, it was shown in \cite{einkemmer2016} that this method can be easily extended to Neumann and mixed boundary conditions. The main disadvantage of this approach is that the correction $q$ has to be computed; although, at least for time-invariant Dirichlet boundary conditions, this has to be done only once at the beginning of the simulation (i.e.~there is no performance penalty).

Second, let us consider the TDBC method which suggests to use a numerical scheme for solving the linear partial flow  (\ref{eq:pf-A}) that satisfies the same boundary condition as the Taylor expansion
\begin{equation} V(s) = v_0 + s Av_0 + \frac{s^2}{2} A^2 v_0. \label{eq:taylor-A} \end{equation}
	Note that restricting the above Taylor expansion to the boundary does \textit{not} yield the original boundary data $b$. It is perhaps not entirely obvious why this method works at all. This can be seen most easily by noting that the Taylor expansion of the Strang splitting algorithm (using (\ref{eq:taylor-A}) to approximate the linear partial flow) yields%
\blue\footnote{\blue As we are working in a framework of smooth solutions and data, the application of the differential operator $A$ is always well defined and bounded. This justifies the use of Landau notation here and in the following.}\black
\[ u + \tau(Au+f(u)) + \tfrac{\tau^2}{2}f^\prime(u)(Au+f(u)) + \tfrac{\tau^2}{2}A(Au + f(u)) + \mathcal{O}(\tau^3), \]
which restricted to the boundary and by using the first and second compatibility condition gives just $b+\mathcal{O}(\tau^3)$ (i.e.~the prescribed boundary condition up to third order accuracy). On the other hand, it is clear that if we enforce the prescribed boundary condition in the middle step of the Strang splitting algorithm, we obtain by restricting the solution to the boundary
\[ b + \tfrac{\tau}{2}f(b) + \tfrac{\tau^2}{8}f^\prime(b)f(b) \]
which, in general, is only first order accurate.

Therefore, the problem at hand is how to incorporate these modified boundary data into a numerical integrator (it should be clear that simply using the Taylor expansion is unstable). To do this we first expand the flow of the non-stiff nonlinearity into a Taylor series
\begin{equation}  v_0 = \varphi_{\tau/2}^f(u_0) = u_0 + \tfrac{\tau}{2}f(u_0) + \tfrac{\tau^2}{8}f^{\prime}(u_0)f(u_0) + \mathcal{O}(\tau^3) \label{eq:taylor-phif} \end{equation}
and substitute this expression into equation (\ref{eq:taylor-A}) in order to obtain (for $s\leq \tau$)
\begin{align}
	V(s)&=u_{0}+\frac{\tau}{2}f_{0}+\frac{\tau^{2}}{8}f_{0}^{\prime}f_{0} +sAu_{0}+\frac{s\tau}{2}Af_{0}+\frac{s^{2}}{2}A^{2}u_{0} + \mathcal{O}(\tau^3), \label{eq:correction-formula-raw}
\end{align}
where we have used $f_0=f(u_0)$ and $f_0^\prime = f^\prime(u_0)$ as a shorthand notation.
Now, this is still not a useful procedure in the sense that in order to obtain a scheme that has local order two (i.e.~neglecting the $\mathcal{O}(\tau^2)$ terms in equation (\ref{eq:correction-formula-raw})) we still have to perform numerical differentiation in order to compute the application of $A$. However, we can use the compatibility \blue conditions
\[ Au \vert_{\partial\Omega} = -f(b),\qquad A^2 u\vert_{\partial\Omega} = -A f(u)\vert_{\partial\Omega}\]
\black to obtain
\begin{align*}
	\overline{V}(s) = b+\frac{\tau}{2}f(b)+\frac{\tau^{2}}{8}f^{\prime}(b)f(b) - sf(b) + \frac{s(\tau-s)}{2}Af_{0} \pO = V(s)\pO + \mathcal{O}(\tau^3). \label{eq:correction-formula}
\end{align*}
This then suggests that instead of equation (\ref{eq:pf-A}) we solve
\begin{equation}\label{eq:pf-A-mod}
\partial_t v = Av,\qquad v\pO = \overline{V}, \qquad v(0)=v_0,
\end{equation}
while no modification is made to equation (\ref{eq:pf-f}).
We will show in section \ref{sec:convergence} that taking \blue all first order terms in $V$, i.e.~choosing
\begin{equation}
	\overline{V}(s) = b+\left(\frac{\tau}{2}-s\right)f(b) = V(s)\pO + \mathcal{O}(\tau^2) \label{eq:pf-Av}
\end{equation}
\black is sufficient to obtain a numerical scheme that is locally and globally second order accurate (under the assumption that the parabolic smoothing property holds true).

Let us now extend this procedure to local order three. Similar to the results obtained for the CEC method, for a differential operator in divergence form $A=\nabla\cdot(a(x){}\cdot{})$ we have
\[ Af(u) \vert_{\partial\Omega} = - f(b)^\prime f(b) + (\nabla\cdot a)(f(b)-f^\prime(b)b). \]
Thus, we can completely determine the correction by using the boundary data and function evaluations of $f$. This, however, is not true for more general operators. For example, for $A=c\Delta$ we get
\[ Af(u)\vert_{\partial\Omega} = c f^{\prime\prime}(b) (\nabla u\pO)^2 - f^\prime(b) f(b). \]
and thus it is still required to evaluate the first derivative of $u$ by numerical differentiation.

The advantage of the present method is that no modification of the partial flow corresponding to the non-stiff reaction is necessary. In addition, no correction in the interior of the domain has to be computed. The disadvantages of this method is that, even for autonomous problems with Dirichlet boundary condition, the linear partial flow has to be solved with a time-dependent boundary condition.

\section{Convergence analysis \label{sec:convergence}}

The purpose of this section is to provide a mathematically rigorous convergence analysis for the TDBC approach. \blue A convergence analysis for the linear case is given in \cite{alonsomallo2016}, another one for the nonlinear Schr\"{o}dinger equation in \cite{canoreguera2017}\black. We will not consider the CEC (compatibility enforcing correction) approach here as the convergence analysis conducted in \cite{einkemmer2015} can be applied immediately to the present case.

For the convergence analysis we will assume that $A_0$, which we use to denote the differential operator $A$ endowed with homogeneous Dirichlet boundary conditions, generates a $\mathcal{C}_0$ semigroup. This, in particular, implies that the time evolution operator $\e^{tA_0}$ is well defined for all $t\geq 0$. The same holds true for $\varphi_k(tA_0)$, where the entire functions $\varphi_k(z)$ are given by the recurrence relation
\begin{equation} \varphi_{k+1}(z) = \frac{\varphi_k(z)-\frac{1}{k!}}{z}, \qquad \varphi_0(z)=\e^z. \label{eq:recurrence} \end{equation}
\blue Problem \eqref{eq:problem} will be studied in a Banach space $X$ with norm $\|\cdot\|$. The nonlinearity $f$ is a function on $X$ and assumed to be sufficiently smooth in a neighborhood of the exact solution. The following situation can be taken as a prototypical example for the whole section. Let $\Omega\in \mathbb R^d$ be a bounded domain with smooth boundary $\partial \Omega$. We consider the abstract parabolic problem \eqref{eq:problem} with $A=\Delta$ in the Hilbert space $X=L^2(\Omega)$. The Laplacian, endowed with homogeneous Dirichlet boundary conditions, will be called $A_0$ and generates an analytic semigroup with domain $H^2(\Omega)\cap H^1_0(\Omega)$.

We repeat once more that we only interested here in the effects of Dirichlet boundary conditions. Therefore, we consider for the analysis a framework of smooth data and solutions. As $A$ is a differential operator (with smooth coefficients), the application of $A$ to (spatially) smooth functions is always well defined and bounded. Note, however, that the application of $A_0$ to a function $g$ requires $g$ to satisfy in addition homogeneous Dirchlet boundary conditions. Therefore, terms involving $A_0$ must be handled with care.

For notational convenience, we will use Landau notation in this section. Note that a sloppy use of this notation caused some misunderstandings in the literature on the numerical analysis of stiff problems. In this paper, however, we will make strict use of Landau notation and include only terms that are reasonably bounded.
\black

Now, let us consider the evolution equation
\begin{equation} v^{\prime}=Av,\qquad v\vert_{\partial\Omega}=\rho \vert_{\partial\Omega}, \qquad v(0)=v_0, \label{eq:evol-A-corr} \end{equation}
where $\rho\blue(t)\black = \rho_0 + t \rho_1$ is assumed to be a smooth function \blue in space\black. In this framework we can write
\begin{align*}
(v-\rho)^{\prime} & =A(v-\rho)-\rho^{\prime}+A\rho\\
 & =A_{0}(v-\rho)+A\rho_{0}-\rho_{1}+tA\rho_{1}.
\end{align*}
To substitute $A_0$ for $A$ is possible since, by construction, $v-\rho$ is zero on the boundary. The solution of (\ref{eq:evol-A-corr}) can then be written as
\[ v(t) = \rho\blue(t)\black+t\varphi_{1}(tA_{0})(A\rho_{0}-\rho_{1})+t^{2}\varphi_{2}(tA_{0})A\rho_{1}. \]
Now we use $\rho_{0}=v_{0}$ (i.e.~that the boundary data given by $\rho$ provides a consistent approximation) to get (assuming that $t\leq\tau$)
\begin{align*}
	v(t) &= v_0 + t\rho_1 + t\varphi_{1}(tA_{0})(Av_0-\rho_{1})+t^{2}\varphi_{2}(tA_{0})A\rho_{1} \nonumber \\
			 &= v_0 + \mathcal{O}(\tau)  \label{eq:v1}.
\end{align*}
		
In general, there is no reason to believe that $Av_0-\rho_1$ lies in the domain of $A_0$. However, if we choose $\rho_1=-f(v_0)$ we have (due to (\ref{eq:taylor-phif}))
\[ Av_{0}-\rho_{1} = Av_{0}+f(v_{0}) = Au_{0}+f(u_{0})+\mathcal{O}(\tau). \]
\blue Note that $Au_0+f(u_0)$ vanishes \black on the boundary due to the first compatibility condition. Thus, we get
\[ v(t) = v_0 - tf(v_0) + t\varphi_{1}(tA_{0})(Au_0+f(u_0)) +
		t \varphi_1(tA_0)m
		- t^2\varphi_{2}(tA_{0})Af(v_0) \]
with
\[ m = A(v_0-u_0) + f(v_0)-f(u_0) = \mathcal{O}(\tau). \]
Further expanding the first $\varphi_1$ function (by using the recurrence relation (\ref{eq:recurrence}) \blue and the first compatibility condition\black) yields
\begin{equation}
	\begin{aligned} \label{eq:v-ord1}
	v(t) &= v_{0}+tAu_{0} + t(f(u_0)-f(v_0)) +t \varphi_1(tA_0)m \\
	& \quad  + t^{2}\varphi_{2}(tA_{0})A(Au_{0}+f(u_0))
	-t^{2}\varphi_{2}(tA_{0})Af(v_0) \\
	\end{aligned}
\end{equation}
and thus
\begin{align}
	v(t) &= v_0 + t Av_0 + \mathcal{O}(\tau^2). \label{eq:v2}
\end{align}

Now, we continue expanding (\ref{eq:v-ord1}) and obtain
\begin{align}
	v(t) & =v_{0}+tAv_{0}+\tfrac{t^{2}}{2}A^2u_0 + \tfrac{t^{2}}{2}A(f(u_{0})-f(v_0))+t^{2}\varphi_{2}(tA_{0})A_{0}m \nonumber \\
		 & \qquad +t^{3}\varphi_{3}(tA_{0})A_{0}A(Au_{0}+f(u_0))-t^{3}\varphi_{3}(tA_{0})A_{0}Af(v_{0}) \nonumber \\
		 & =v_{0}+tAv_{0}+\tfrac{t^{2}}{2}A^{2}v_{0} + t^2(\tau\varphi_2(tA_0)A_0\tilde{E} + t\varphi_3(tA_0)A_0\hat{E}) + \mathcal{O}(\tau^3). \label{eq:v3}
\end{align}
Before proceeding, let us note that while $\tilde{E}$ and $\hat{E}$ are bounded they do not lie in the domain of $A_0$. Thus, we can not simply absorb this part of the remainder into the $\mathcal{O}(\tau^3)$ term. \blue Nevertheless, \eqref{eq:v3} is well defined since $t\varphi_2(tA_0)A_0$ and $t\varphi_3(tA_0)A_0$ are bounded operators on $X$. We will see later by using the parabolic smoothing property that these two terms in \eqref{eq:v3} can be appropriately \black bounded as part of the convergence analysis. For now, the results obtained allow us to formulate the following theorem for the local error. Before stating the theorem let us note that we use $u(t)$ to denote the exact solution of equation (\ref{eq:problem}).

\begin{theorem} Let us assume that $A_0$ generates a $\mathcal{C}_0$ semigroup, that $f$ is twice differentiable, and that $u(0)$ is sufficiently smooth. Then performing the second order TDBC correction \blue \rm{(}i.e.~Strang splitting based on \eqref{eq:pf-f}, \eqref{eq:pf-A-mod}, and \eqref{eq:pf-Av}\rm{)} \black results in a numerical scheme for which the local error can be bounded as follows
	\[ \Vert S_{\tau}(u(t_n))-u(t_n+\tau) \Vert \leq C \tau^2, \]
	where $C$ is independent of $\tau$. In addition, \blue if $A$ generates an analytic semigroup \black we have
	\[ S_{\tau}(u(t_n))-u(t_n+\tau) =  \tau^3 A_0\overline{E} + \mathcal{O}(\tau^3), \]
	where $\overline{E}$ satisfies $\Vert \overline{E} \Vert \leq C$ and $\Vert A_0 \overline{E} \Vert \leq C/\tau$ with $C$ independent of $\tau$. \label{thm:consistency}
\end{theorem}
\begin{proof}
Let us compute the local error of the splitting scheme. To start, we have
\[ v_0 = \varphi_{\tau/2}^f(u_0) = u_{0}+\frac{\tau}{2}f(u_{0})+\frac{\tau^{2}}{8}f^{\prime}(u_{0})f(u_{0}) + \mathcal{O}(\tau^3).
\]
Now, we use equation (\ref{eq:v2}) to obtain
\begin{equation} v(\tau) = u_0 + \frac{\tau}{2} f(u_0) + \tau A u_0 + \mathcal{O}(\tau^2) \label{eq:vtau-1} \end{equation}
or equation (\ref{eq:v3}) to obtain
\begin{equation}
	\begin{aligned}
		v(\tau)&=u_{0}+\frac{\tau}{2}f(u_{0})+\frac{\tau^{2}}{8}f^{\prime}(u_{0})f(u_{0})+\tau Au_{0} \\
						& \quad+\frac{\tau^{2}}{2}Af(u_{0})+\frac{\tau^{2}}{2}A^{2}u_{0}+\tau^{3}A_0\overline{E}+\mathcal{O}(\tau^{3}).
	\end{aligned} \label{eq:vtau-2}
\end{equation}
	Note that $\overline{E}$ is bounded. \blue If $A_0$ is the generator of an analytic semigroup, it holds that $\Vert \varphi_k(\tau A_0)A_0\Vert \leq C/\tau$ for $k\ge 1$. In this situation\black, we obtain the following bound for the remainder term $\Vert A_0 \overline{E}\Vert \leq C/\tau$.

Finally, we use $w_0=v(\tau)$ and equation (\ref{eq:vtau-1}) to obtain
\begin{align*}
	w(\tau) &= w_0 + \frac{\tau}{2}f(w_0) + \mathcal{O}(\tau^2) \\
				 &= u_0 + \tau(Au_0 + f(u_0)) + \mathcal{O}(\tau^2)
\end{align*}
and equations (\ref{eq:vtau-1}) and (\ref{eq:vtau-2}) to obtain
\begin{align*}
	w(\tau) &= w_{0}+\frac{\tau}{2}f(w_{0})+\frac{\tau^{2}}{8}f^{\prime}(w_{0})f(w_{0}) + \mathcal{O}(\tau^3)\\
 & =u_{0}+\tau(Au_{0}+f(u_{0}))+\frac{\tau^{2}}{2}\left(A(Au_{0}+f(u_{0}))+f^{\prime}(u_{0})(Au_{0}+f(u_{0}))\right) \\
	& \qquad + \tau^{3}A_0\overline{E} + \mathcal{O}(\tau^{3}).
\end{align*}
The results obtained above can be compared to the expansion of the exact solution
\[ u(\tau)=u_{0}+\tau(Au_{0}+f(u_{0}))+\frac{\tau^{2}}{2}\bigl(A(Au_{0}+f(u_{0}))+f^{\prime}(u_{0})(Au_{0}+f(u_{0}))\bigr)+\mathcal{O}\left(\tau^{3}\right) \]
which immediately yields the desired expressions for the local error.
\end{proof}

Now let us show that this result is in fact sufficient to obtain global convergence of order two in the case of analytic semigroups. In the following we will use $u_n$ to denote the numerical approximation at time $t_n=n\tau$.

\begin{theorem} Let us assume that $A_0$ generates an analytic semigroup, that $f$ is twice differentiable, and that $u(0)$ is sufficiently smooth. Then performing the second order TDBC correction \blue \rm{(}i.e.~Strang splitting based on \eqref{eq:pf-f}, \eqref{eq:pf-A-mod}, and \eqref{eq:pf-Av}\rm{)} \black results in a numerical scheme that is second order convergent, i.e.
	\[ \Vert u_{n}-u(t_{n}) \Vert\leq C\tau^{2}(1+\left|\log\tau\right|) \]
	for all $t_n=n\tau\leq T$ with $C$ independent of $\tau$ and $n$.
\end{theorem}
\begin{proof}
First we define the global error
\[ e_{n}=u_{n}-u(t_{n}) \]
and cast it into the following form
\begin{align*}
	e_{n+1} & =S_{\tau}(u_{n})-S_{\tau}(u(t_{n}))+d_{n+1}\\
							 & =\varphi_{\tau/2}^{f}\circ\varphi_{\tau}^{A}\circ\varphi_{\tau/2}^f(u_{n})
								 -\varphi_{\tau/2}^{f}\circ\varphi_{\tau}^{A}\circ\varphi_{\tau/2}^{f}(u(t_{n}))
								 +d_{n+1},
\end{align*}
where \blue $\varphi_\tau^A(v_0)$ denotes the exact solution of \eqref{eq:pf-A-mod}, \eqref{eq:pf-Av}. The defect $d_{n+1}$ \black is given by
\[ d_{n+1} = S_{\tau}(u(t_n))-u(t_n+\tau). \]

Now, let us introduce $\beta$ such that \blue $\beta\pO=\overline V$ as defined in \eqref{eq:pf-Av} \black and $A\beta=0$. Then we can write
\[ \varphi_{\tau}^A(z) = \e^{tA_0}(z-\beta) + \beta \]
and consequently
\begin{equation} \varphi_{\tau}^{A}\circ\varphi_{\tau/2}^f(u_{n}) - \varphi_{\tau}^{A}\circ\varphi_{\tau/2}^{f}(u(t_{n}))
= \mathrm{e}^{\tau A_0}D(u_n,u(t_n)), \label{eq:linearity} \end{equation}
where
\[ D(u_n,u(t_n)) = \varphi_{\tau/2}^f(u_{n})-\varphi_{\tau/2}^{f}(u(t_{n})). \]
Since $\varphi^f_{\blue\tau/2\black}(z) = z + \tau H(z)$ \blue holds for some Lipschitz continuous function $H$\black, equation (\ref{eq:linearity}) yields
\begin{equation} e_{n+1} = \e^{\tau A_0}D(u_n,u(t_n)) + \tau H(\varphi_{\tau}^{A}\circ\varphi_{\tau/2}^f(u_{n}))-\tau H(\varphi_{\tau}^{A}\circ\varphi_{\tau/2}^{f}(u(t_{n}))) + d_{n+1}. \label{eq:error1} \end{equation}

Now, we use
\[ D(u_n,u(t_n)) = e_n + \tau (H(u_n)-H(u(t_n)) \]
which substituted into equation (\ref{eq:error1}) yields
\begin{equation} e_{n+1} = \e^{\tau A_0}e_n + \tau E_{n} + d_{n+1}. \label{eq:error-recursion} \end{equation}
	Note that we can bound $E_{n}$ as follows $\Vert E_{n}\Vert \leq C \Vert e_{n} \Vert$.

Solving the linear part of the recurrence relation (\ref{eq:error-recursion}) we get
\[
	e_{n}=\e^{n\tau A_0}e_{0}+\sum_{k=1}^{n}\e^{(n-k)\tau A_0}d_{k}+\tau\sum_{k=0}^{n-1}\e^{(n-k-1)\tau A_0}E_{k}.
\]
Now, in the light of Theorem \ref{thm:consistency}, the terms $d_k$ are \blue of the form $d_k = \tau^3 A_0 \overline E + \mathcal{O}(\tau^3)$. This, in general\black, is not sufficient to obtain convergence of order two. However, by the parabolic smoothing property (which is a consequence of the assumption that $A_0$ generates an analytic semigroup)
\[
	\Vert e^{tA_0}A_0 \Vert\leq \blue \frac{C}{t},\quad 0<t\le T\black
\]
we get
\[
\Vert e_{n}\Vert\leq C\Vert e_{0}\Vert+C\tau^{3}\sum_{k=1}^{n-1}\frac{1}{k\tau}+C\tau\sum_{k=0}^{n-1}\Vert e_{k}\Vert.
\]
Then applying Gronwall's inequality (with $\Vert e_{0}\Vert=0$) yields
\[
\Vert e_{n}\Vert\leq C\tau^{2}(1+\left|\log\tau\right|),
\]
which is the desired result.
\end{proof}

A corollary of the above calculation is that the unmodified Strang splitting (i.e.~without performing any correction) is of order one if a reaction $f$ with $f(b)\neq b$ is used. If $f(b)=b$ then we get locally order two and globally (assuming the parabolic smoothing property) order two. Finally, if $f(b)=b$ and we have $Af(u)\pO=b$ for all $u$ with $u\pO=b$, then we get locally order three and globally order two (even without the parabolic smoothing property).

\section{Numerical results \label{sec:numerics}}

We now turn to a number of numerical examples that (in addition to confirming the theoretical results obtained) are used to investigate the relative accuracy of the TDBC and CEC correction approach. In addition, our goal is to investigate under what circumstances performing the third order correction increases the accuracy compared to the second order correction. \blue Although our theoretical convergence results are valid in any dimension, we use the interval $[0,1]$ as the computational domain in all simulations. This is sufficient since the observed order reduction does not depend on the dimension of the problem. Further, we restrict ourselves to time-invariant boundary conditions, since time dependent boundary data will not behave in a different way (see also \cite{einkemmer2015}).  Note that \black all numerical approximations are compared to a reference solution obtained by specifying a tolerance of $10^{-14}$ for a traditional (i.e.~unsplit) time integrator. Note that all errors reported are measured in the maximum norm.

\subsection{Parabolic problem}

We consider the parabolic diffusion-reaction equation
\begin{equation} \partial_t u(t,x) = \partial_{xx}u(t,x) + f(u(t,x),x), \qquad u\pO=b, \qquad u(0,x)=u_0(x). \label{eq:diffusion} \end{equation}
	The spatial derivative is discretized using standard centered finite differences.

	As the first experiment we use $b=0$ and employ three different reaction terms of the form $f(u,x)=u+1$, $f(u,x)=u+p(x)$, and $f(u,x)=u+q(x)$ with $p(x)=x(1-x)$ and $q(x)=u + (x-1)x(-1 - x + x^2)$. These reaction terms are constructed such that for the unmodified splitting (i.e.~without performing any correction), according to the analysis given in section \ref{sec:convergence} and \cite{einkemmer2015}, we expect a local error of first, second, and third order and a global error of first, second, and second order, respectively. The numerical results shown in Table \ref{tab:diff-different-f} confirm this behavior.
\begin{table}[htb]
	\begin{center}
	\textbf{Local error}
	\vspace{0.1cm}

	\begin{tabular}{rrrrrrrrr}
		& \multicolumn{2}{c}{$f(u)=u+1$} && \multicolumn{2}{c}{$f(u)=u+p(x)$} && \multicolumn{2}{c}{$f(u)=u+q(x)$} \\
		\cline{2-3} \cline{5-6} \cline{8-9}
		\multicolumn{1}{l}{step size} & \multicolumn{1}{l}{$l^{\infty}$ error} & \multicolumn{1}{l}{order} & &
		\multicolumn{1}{l}{$l^{\infty}$ error} & \multicolumn{1}{l}{order} & & \multicolumn{1}{l}{$l^{\infty}$ error} & \multicolumn{1}{l}{order} \\
		\hline
		6.40e-02    & 3.14e-02    & --     &&  4.08e-04    & --   && 4.54e-04    & --      \\
		3.20e-02    & 1.54e-02    & 1.03  &&  9.93e-05    & 2.04  && 6.13e-05    & 2.89  \\
		1.60e-02    & 7.51e-03    & 1.03  &&  2.48e-05    & 2.00  && 7.72e-06    & 2.99  \\
		8.00e-03    & 3.64e-03    & 1.04  &&  6.21e-06    & 2.00  && 9.69e-07    & 2.99  \\
		4.00e-03    & 1.75e-03    & 1.06  &&  1.55e-06    & 2.00  && 1.22e-07    & 2.99  \\
		2.00e-03    & 8.24e-04    & 1.08  &&  3.88e-07    & 2.00  && 1.54e-08    & 2.99  \\
	\end{tabular}
	\vspace{0.25cm}

	\textbf{Global error}
	\vspace{0.1cm}

	\begin{tabular}{rrrrrrrrr}
		& \multicolumn{2}{c}{$f(u)=u+1$} && \multicolumn{2}{c}{$f(u)=u+p(x)$}  && \multicolumn{2}{c}{$f(u)=u+q(x)$} \\
		\cline{2-3} \cline{5-6} \cline{8-9}
		\multicolumn{1}{l}{step size} & \multicolumn{1}{l}{$l^{\infty}$ error} & \multicolumn{1}{l}{order} & &
		\multicolumn{1}{l}{$l^{\infty}$ error} & \multicolumn{1}{l}{order} && \multicolumn{1}{l}{$l^{\infty}$ error} & \multicolumn{1}{l}{order} \\
		\hline
		6.40e-02    & 3.15e-02    & --    &&  6.75e-04    & --  &&  9.66e-04  & --     \\
		3.20e-02    & 1.54e-02    & 1.03 &&  1.71e-04    & 1.98 &&  2.41e-04  & 2.00  \\
		1.60e-02    & 7.52e-03    & 1.03 &&  4.34e-05    & 1.98 &&  6.01e-05  & 2.00  \\
		8.00e-03    & 3.65e-03    & 1.04 &&  1.09e-05    & 1.99 &&  1.50e-05  & 2.00  \\
		4.00e-03    & 1.75e-03    & 1.06 &&  2.75e-06    & 1.99 &&  3.76e-06  & 2.00  \\
		2.00e-03    & 8.29e-04    & 1.08 &&  6.91e-07    & 1.99 &&  9.40e-07  & 2.00  \\
	\end{tabular}
	\end{center}

	\caption{The local (at $t=0$) and global errors using the unmodified Strang splitting applied to equation (\ref{eq:diffusion}) are shown. The three different reaction terms indicated in the text are used. The space discretization is conducted by using the standard centered finite difference stencil with $200$ grid points. All problems are integrated until $t=0.25$ and use the initial value $u(0,x)=0$.\label{tab:diff-different-f}}
\end{table}

Now, let us investigate the two corrections which yield a second order accurate splitting scheme. In this case we set $b=1$ and $f(u)=\e^{u-1}$. The numerical results are shown in Table \ref{tab:diffusion-2nd} and confirm the convergence analysis conducted in section \ref{sec:convergence} and \cite{einkemmer2015}. The error of both corrected splittings (CEC and TDBC), even for medium precision requirements, is superior by almost three orders of magnitude compared to the unmodified Strang splitting. We also observe that in this case the error of the TDBC approach is smaller by approximately 40\% compared to the CEC approach.

\begin{table}[htb]
	\begin{center}
	\textbf{Local error}
	\vspace{0.1cm}

	\begin{tabular}{rrrrrrrrr}
		& \multicolumn{2}{c}{unmodified} && \multicolumn{2}{c}{TDBC} && \multicolumn{2}{c}{CEC} \\
		\cline{2-3} \cline{5-6} \cline{8-9}
		\multicolumn{1}{l}{step size} & \multicolumn{1}{l}{$l^{\infty}$ error} & \multicolumn{1}{l}{order} & &
		\multicolumn{1}{l}{$l^{\infty}$ error} & \multicolumn{1}{l}{order} & & \multicolumn{1}{l}{$l^{\infty}$ error} & \multicolumn{1}{l}{order} \\
		\hline
		1.60e-02    & 7.49e-03    & --     && 1.25e-04    & --     && 1.06e-04    & --      \\
		8.00e-03    & 3.64e-03    & 1.04   && 3.25e-05    & 1.94   && 2.76e-05    & 1.94    \\
		4.00e-03    & 1.75e-03    & 1.06   && 8.17e-06    & 1.99   && 6.91e-06    & 2.00    \\
		2.00e-03    & 8.24e-04    & 1.08   && 2.04e-06    & 2.00   && 1.73e-06    & 2.00    \\
		1.00e-03    & 3.79e-04    & 1.12   && 5.13e-07    & 2.00   && 4.31e-07    & 2.00    \\
		5.00e-04    & 1.68e-04    & 1.18   && 1.27e-07    & 2.01   && 1.07e-07    & 2.00    \\
	\end{tabular}
	\vspace{0.25cm}

	\textbf{Global error}
	\vspace{0.1cm}

	\begin{tabular}{rrrrrrrrr}
		& \multicolumn{2}{c}{unmodified} && \multicolumn{2}{c}{TDBC}  && \multicolumn{2}{c}{CEC} \\
		\cline{2-3} \cline{5-6} \cline{8-9}
		\multicolumn{1}{l}{step size} & \multicolumn{1}{l}{$l^{\infty}$ error} & \multicolumn{1}{l}{order} & &
		\multicolumn{1}{l}{$l^{\infty}$ error} & \multicolumn{1}{l}{order} && \multicolumn{1}{l}{$l^{\infty}$ error} & \multicolumn{1}{l}{order} \\
		\hline
		1.60e-02    & 7.52e-03    & --     && 3.13e-05    & --     && 4.15e-05    & --    \\
		8.00e-03    & 3.65e-03    & 1.04   && 7.72e-06    & 2.02   && 1.04e-05    & 2.00  \\
		4.00e-03    & 1.75e-03    & 1.06   && 1.91e-06    & 2.02   && 2.60e-06    & 2.00  \\
		2.00e-03    & 8.29e-04    & 1.08   && 4.69e-07    & 2.02   && 6.49e-07    & 2.00  \\
		1.00e-03    & 3.82e-04    & 1.12   && 1.15e-07    & 2.03   && 1.62e-07    & 2.00  \\
		5.00e-04    & 1.70e-04    & 1.17   && 2.81e-08    & 2.03   && 4.06e-08    & 2.00  \\
	\end{tabular}
	\end{center}

	\caption{The local (at $t=0$) and global errors for the unmodified Strang splitting as well as the second order TDBC and CEC corrected Strang splitting applied to equation (\ref{eq:diffusion}) with $f(u)=\e^{u-1}$ are shown. The space discretization is conducted by using the standard centered finite difference stencil with $200$ grid points. All problems are integrated until $t=0.25$ and use the initial value $u(0,x)=\sin \pi x$.
 \label{tab:diffusion-2nd}}
\end{table}

Now let us consider the third order correction. As has been outlined in section \ref{sec:corrections}, the resulting numerical scheme is locally third order accurate but requires numerical differentiation in order to compute the first derivative of $u$ necessary for the correction. The numerical results are shown in Table \ref{tab:diffusion-3rd}. We observe that it is \textit{not} advantageous to employ this correction as the overall global error is slightly larger compared to the second order correction. This is true for both the TDBC and the CEC correction.

\begin{table}[htb]
	\begin{center}
	\textbf{Local error}
	\vspace{0.1cm}

	\begin{tabular}{rrrrrrrrr}
		& \multicolumn{2}{c}{unmodified} && \multicolumn{2}{c}{TDBC} && \multicolumn{2}{c}{CEC} \\
		\cline{2-3} \cline{5-6} \cline{8-9}
		\multicolumn{1}{l}{step size} & \multicolumn{1}{l}{$l^{\infty}$ error} & \multicolumn{1}{l}{order} & &
		\multicolumn{1}{l}{$l^{\infty}$ error} & \multicolumn{1}{l}{order} & & \multicolumn{1}{l}{$l^{\infty}$ error} & \multicolumn{1}{l}{order} \\
		\hline
		1.60e-02    & 7.49e-03    & --     && 1.71e-04    & --    && 8.81e-05    & --      \\
		8.00e-03    & 3.64e-03    & 1.04   && 1.86e-05    & 3.2   && 1.44e-05    & 2.61    \\
		4.00e-03    & 1.75e-03    & 1.06   && 2.29e-06    & 3.02  && 2.11e-06    & 2.77    \\
		2.00e-03    & 8.24e-04    & 1.08   && 3.11e-07    & 2.88  && 2.87e-07    & 2.88    \\
		1.00e-03    & 3.79e-04    & 1.12   && 4.06e-08    & 2.94  && 3.75e-08    & 2.94    \\
		5.00e-04    & 1.68e-04    & 1.18   && 5.18e-09    & 2.97  && 4.80e-09    & 2.97    \\
	\end{tabular}
	\vspace{0.25cm}

	\textbf{Global error}
	\vspace{0.1cm}

	\begin{tabular}{rrrrrrrrr}
		& \multicolumn{2}{c}{unmodified} && \multicolumn{2}{c}{TDBC}  && \multicolumn{2}{c}{CEC} \\
		\cline{2-3} \cline{5-6} \cline{8-9}
		\multicolumn{1}{l}{step size} & \multicolumn{1}{l}{$l^{\infty}$ error} & \multicolumn{1}{l}{order} & &
		\multicolumn{1}{l}{$l^{\infty}$ error} & \multicolumn{1}{l}{order} && \multicolumn{1}{l}{$l^{\infty}$ error} & \multicolumn{1}{l}{order} \\
		\hline
		1.60e-02    & 7.52e-03    & --    && 2.32e-04    & --    && 6.85e-05    & --       \\
		8.00e-03    & 3.65e-03    & 1.04  && 3.30e-05    & 2.81  && 1.67e-05    & 2.04     \\
		4.00e-03    & 1.75e-03    & 1.06  && 5.89e-06    & 2.49  && 4.11e-06    & 2.02     \\
		2.00e-03    & 8.29e-04    & 1.08  && 1.22e-06    & 2.27  && 1.02e-06    & 2.01     \\
		1.00e-03    & 3.82e-04    & 1.12  && 2.77e-07    & 2.14  && 2.54e-07    & 2.00     \\
		5.00e-04    & 1.70e-04    & 1.17  && 6.59e-08    & 2.07  && 6.34e-08    & 2.00     \\
	\end{tabular}
	\end{center}

	\caption{The local (at $t=0$) and global errors for the unmodified Strang splitting as well as the third order TDBC and CEC corrected Strang splitting applied to equation (\ref{eq:diffusion}) with $f(u)=\e^{u-1}$ are shown. The space discretization is conducted by using the standard centered finite difference stencil with $200$ grid points. All problems are integrated until $t=0.25$ and use the initial value $u(0,x)=\sin \pi x$.
 \label{tab:diffusion-3rd}}
\end{table}

\subsection{Hyperbolic problems}

First, we consider the simple advection-reaction equation
\begin{equation} \partial_t u(t,x) = \partial_x u(t,x) + f(u(t,x),x), \qquad u(t,0)=b, \qquad u(0,x)=u_0(x). \label{eq:advection} \end{equation}
	Since $\partial_x$ and $u\mapsto f(u)$ commute, the error is only due to the boundary (assuming that no error is made in the integration of the two partial flows). Note that in this case we only prescribe boundary data at the \blue inflow boundary which is the \black left endpoint of the domain. We choose $b=0$ and employ three different reaction terms that according to the convergence analysis conducted should give a local error of order one, two, and three, respectively. The situation with respect to the global error is more complex as in the present no parabolic smoothing is available. The only fact we can therefore reduce from theory is that the global convergence order should not be below zero, one, and two, respectively (and obviously can not exceed two). The corresponding numerical results are shown in Table \ref{tab:advection-threef}. The results for the local error agree very well with the theoretical prediction. For the global error we observe order one, two, and two, respectively. That is, even though there is no parabolic smoothing the order is in fact identical to what we would expect based on the convergence analysis conducted for the parabolic case. We postpone the explanation of this behavior \blue towards the \black end of this section, \blue where we discuss a slightly more general problem.\black

\begin{table}[htb]
	\begin{center}
	\textbf{Local error}
	\vspace{0.1cm}

	\begin{tabular}{rrrrrrrrr}
		& \multicolumn{2}{c}{$f(u)=u+1$} && \multicolumn{2}{c}{$f(u)=u+x$} && \multicolumn{2}{c}{$f(u)=u+x^2$} \\
		\cline{2-3} \cline{5-6} \cline{8-9}
		\multicolumn{1}{l}{step size} & \multicolumn{1}{l}{$l^{\infty}$ error} & \multicolumn{1}{l}{order} & &
		\multicolumn{1}{l}{$l^{\infty}$ error} & \multicolumn{1}{l}{order} & & \multicolumn{1}{l}{$l^{\infty}$ error} & \multicolumn{1}{l}{order} \\
		\hline
		2.40e-01    & 1.26e-01    & --      & & 7.70e-03    & --    & & 2.41e-03    & --      \\
		1.20e-01    & 6.08e-02    & 1.05    & & 1.84e-03    & 2.07  & & 2.94e-04    & 3.04    \\
		6.00e-02    & 2.94e-02    & 1.05    & & 4.44e-04    & 2.05  & & 3.63e-05    & 3.02    \\
		3.00e-02    & 1.41e-02    & 1.06    & & 1.06e-04    & 2.06  & & 4.51e-06    & 3.01    \\
		1.50e-02    & 6.53e-03    & 1.11    & & 2.47e-05    & 2.10  & & 5.62e-07    & 3.00    \\
		7.50e-03    & 2.76e-03    & 1.24    & & 5.45e-06    & 2.18  & & 6.98e-08    & 3.01    \\
	\end{tabular}
	\vspace{0.25cm}

	\textbf{Global error}
	\vspace{0.1cm}

	\begin{tabular}{rrrrrrrrr}
		& \multicolumn{2}{c}{$f(u)=u+1$} && \multicolumn{2}{c}{$f(u)=u+x$}  && \multicolumn{2}{c}{$f(u)=u+x^2$} \\
		\cline{2-3} \cline{5-6} \cline{8-9}
		\multicolumn{1}{l}{step size} & \multicolumn{1}{l}{$l^{\infty}$ error} & \multicolumn{1}{l}{order} & &
		\multicolumn{1}{l}{$l^{\infty}$ error} & \multicolumn{1}{l}{order} && \multicolumn{1}{l}{$l^{\infty}$ error} & \multicolumn{1}{l}{order} \\
		\hline
		2.40e-01    & 1.25e-01    & --       & & 1.11e-02    & --   && 1.14e-02    & --       \\
		1.20e-01    & 5.98e-02    & 1.07     & & 2.16e-03    & 2.36 && 2.92e-03    & 1.96     \\
		6.00e-02    & 2.85e-02    & 1.07     & & 4.44e-04    & 2.28 && 7.32e-04    & 1.99     \\
		3.00e-02    & 1.31e-02    & 1.12     & & 1.06e-04    & 2.06 && 1.83e-04    & 2.00     \\
		1.50e-02    & 5.54e-03    & 1.24     & & 2.55e-05    & 2.06 && 4.56e-05    & 2.00     \\
		7.50e-03    & 1.94e-03    & 1.51     & & 6.37e-06    & 2.00 && 1.14e-05    & 2.00     \\
	\end{tabular}
	\end{center}

	\caption{The local (at $t=0$) and global errors using the unmodified Strang splitting applied to equation (\ref{eq:advection}) are shown for the three different reaction terms indicated in the table. The space discretization is conducted by using a second order upwind finite difference stencil with $10^3$ grid points. All problems are integrated until $t=1.9$ and use the initial value $u(0,x)=0$.	
  \label{tab:advection-threef}}
\end{table}

Now, let us turn our attention to the two corrections. In order to provide a more realistic example, we will use the following equation
\begin{equation} \partial_t u(t,x) = \partial_x(a(x)u(t,x)) + f(u(t,x)), \qquad u(t,0)=1, \qquad u(0,x)=u_0(x), \label{eq:conservation-law} \end{equation}
	where, if not indicated otherwise, we use $a(x)=1+\sin x$.
	Note that in this case the two operators do not commute and thus we will observe error propagation in the interior of the domain. The numerical results in Table \ref{tab:advection-2nd} compare the unmodified Strang splitting with the TDBC and the CEC corrected Strang splitting. For both corrected versions the error, even for low precision requirements, is almost an order of magnitude smaller than for the uncorrected case. In addition, we observe that the error for the CEC correction is smaller by at least a factor of $6$ compared to the TDBC correction. Thus, in this particular example the CEC method has a significant advantage.

\begin{table}[htb]
	\begin{center}
	\textbf{Local error}
	\vspace{0.1cm}

	\begin{tabular}{rrrrrrrrr}
		& \multicolumn{2}{c}{unmodified} && \multicolumn{2}{c}{TDBC} && \multicolumn{2}{c}{CEC} \\
		\cline{2-3} \cline{5-6} \cline{8-9}
		\multicolumn{1}{l}{step size} & \multicolumn{1}{l}{$l^{\infty}$ error} & \multicolumn{1}{l}{order} & &
		\multicolumn{1}{l}{$l^{\infty}$ error} & \multicolumn{1}{l}{order} & & \multicolumn{1}{l}{$l^{\infty}$ error} & \multicolumn{1}{l}{order} \\
		\hline
		2.40e-01    & 1.25e-01    & --     && 1.51e-02    & --    && 8.80e-03    & --       \\
		1.20e-01    & 5.98e-02    & 1.07   && 2.14e-03    & 2.82  && 1.93e-03    & 2.19     \\
		6.00e-02    & 2.84e-02    & 1.07   && 4.73e-04    & 2.18  && 4.42e-04    & 2.13     \\
		3.00e-02    & 1.31e-02    & 1.12   && 1.15e-04    & 2.04  && 1.01e-04    & 2.13     \\
		1.50e-02    & 5.53e-03    & 1.24   && 2.84e-05    & 2.02  && 2.20e-05    & 2.19     \\
		7.50e-03    & 1.89e-03    & 1.55   && 6.91e-06    & 2.04  && 4.68e-06    & 2.24     \\
	\end{tabular}
	\vspace{0.25cm}

	\textbf{Global error}
	\vspace{0.1cm}

	\begin{tabular}{rrrrrrrrr}
		& \multicolumn{2}{c}{unmodified} && \multicolumn{2}{c}{TDBC}  && \multicolumn{2}{c}{CEC} \\
		\cline{2-3} \cline{5-6} \cline{8-9}
		\multicolumn{1}{l}{step size} & \multicolumn{1}{l}{$l^{\infty}$ error} & \multicolumn{1}{l}{order} & &
		\multicolumn{1}{l}{$l^{\infty}$ error} & \multicolumn{1}{l}{order} && \multicolumn{1}{l}{$l^{\infty}$ error} & \multicolumn{1}{l}{order} \\
		\hline
		2.40e-01    & 3.73e-01    & --      && 1.09e-01    & --   && 2.59e-02    & --     \\
		1.20e-01    & 9.07e-02    & 2.04    && 3.56e-02    & 1.62 && 4.72e-03    & 2.46   \\
		6.00e-02    & 2.84e-02    & 1.67    && 1.02e-02    & 1.80  && 1.30e-03    & 1.86   \\
		3.00e-02    & 1.31e-02    & 1.12    && 2.74e-03    & 1.90  && 4.15e-04    & 1.65   \\
		1.50e-02    & 5.54e-03    & 1.24    && 7.07e-04    & 1.95 && 1.16e-04    & 1.83   \\
		7.50e-03    & 1.94e-03    & 1.51    && 1.80e-04    & 1.98 && 3.08e-05    & 1.92   \\
	\end{tabular}
	\end{center}

	\caption{The local (at $t=0$) and global errors for the unmodified Strang splitting as well as the second order TDBC and CEC corrected Strang splitting applied to equation (\ref{eq:conservation-law}) with $f(u)=\e^{u-1}$ are shown. The space discretization is conducted by using a second order upwind finite difference stencil with $500$ grid points. All problems are integrated until $t=1.9$ and use the initial value $u(0,x)=1+x$.
		 \label{tab:advection-2nd}}
\end{table}

Now, we perform an identical numerical experiment except for the fact that we use the third order TDBC and CEC corrections. The numerical results are shown in Table \ref{tab:advection-3rd}. If we compare these results with the second order corrections in Table \ref{tab:advection-2nd} we find that for the TDBC approach the third order correction reduces the error by approximately 50\%, while for the CEC correction the error increases by approximately 20\%. Nevertheless, the second order CEC correction is still the most accurate method overall. This is despite the fact that in this example no numerical differentiation is required in order to evaluate the correction.

\begin{table}[htb]
	\begin{center}
	\textbf{Local error}
	\vspace{0.1cm}

	\begin{tabular}{rrrrrrrrr}
		& \multicolumn{2}{c}{unmodified} && \multicolumn{2}{c}{TDBC} && \multicolumn{2}{c}{CEC} \\
		\cline{2-3} \cline{5-6} \cline{8-9}
		\multicolumn{1}{l}{step size} & \multicolumn{1}{l}{$l^{\infty}$ error} & \multicolumn{1}{l}{order} & &
		\multicolumn{1}{l}{$l^{\infty}$ error} & \multicolumn{1}{l}{order} & & \multicolumn{1}{l}{$l^{\infty}$ error} & \multicolumn{1}{l}{order} \\
		\hline
		2.40e-01    & 1.25e-01    & --     && 1.51e-02    & --     && 1.39e-02    & --    \\
		1.20e-01    & 5.98e-02    & 1.07   && 2.14e-03    & 2.82   && 1.61e-03    & 3.11  \\
		6.00e-02    & 2.84e-02    & 1.07   && 2.87e-04    & 2.90   && 1.90e-04    & 3.09  \\
		3.00e-02    & 1.31e-02    & 1.12   && 3.72e-05    & 2.95   && 2.28e-05    & 3.06  \\
		1.50e-02    & 5.53e-03    & 1.24   && 4.73e-06    & 2.97   && 2.79e-06    & 3.03  \\
		7.50e-03    & 1.89e-03    & 1.55   && 5.98e-07    & 2.99   && 3.44e-07    & 3.02  \\
	\end{tabular}
	\vspace{0.25cm}

	\textbf{Global error}
	\vspace{0.1cm}

	\begin{tabular}{rrrrrrrrr}
		& \multicolumn{2}{c}{unmodified} && \multicolumn{2}{c}{TDBC} && \multicolumn{2}{c}{CEC} \\
		\cline{2-3} \cline{5-6} \cline{8-9}
		\multicolumn{1}{l}{step size} & \multicolumn{1}{l}{$l^{\infty}$ error} & \multicolumn{1}{l}{order} & &
		\multicolumn{1}{l}{$l^{\infty}$ error} & \multicolumn{1}{l}{order} && \multicolumn{1}{l}{$l^{\infty}$ error} & \multicolumn{1}{l}{order} \\
		\hline
		2.40e-01    & 3.73e-01    & --     && 5.46e-02    & --    && 4.65e-02    & --     \\
		1.20e-01    & 9.07e-02    & 2.04   && 2.04e-02    & 1.42  && 1.08e-02    & 2.10    \\
		6.00e-02    & 2.84e-02    & 1.67   && 6.25e-03    & 1.70   && 2.57e-03   & 2.08   \\
		3.00e-02    & 1.31e-02    & 1.12   && 1.73e-03    & 1.85  && 6.24e-04    & 2.05   \\
		1.50e-02    & 5.54e-03    & 1.24   && 4.55e-04    & 1.93  && 1.53e-04    & 2.03   \\
		7.50e-03    & 1.94e-03    & 1.51   && 1.17e-04    & 1.96  && 3.79e-05    & 2.01   \\
	\end{tabular}
	\end{center}

	\caption{The local (at $t=0$) and global errors for the unmodified Strang splitting as well as the third order TDBC and CEC corrected Strang splitting applied to equation (\ref{eq:conservation-law}) with $f(u)=\e^{u-1}$ are shown. The space discretization is conducted by using a second order upwind finite difference stencil with $500$ grid points. All problems are integrated until $t=1.9$ and use the initial value $u(0,x)=1+x$.
	\label{tab:advection-3rd}}
\end{table}

\blue We still have to explain \black the fact that, in all simulations conducted so far, we actually observe the same order locally as well as globally. For the convergence analysis conducted in section \ref{sec:convergence} and \cite{einkemmer2015} this requires the parabolic smoothing property which is not applicable to the hyperbolic equation we consider in this section. \blue For a better understanding of the situation, we repeat the experiment of Table \ref{tab:advection-2nd} but measure now the error away from the inflow boundary in the interval $[\tfrac{1}{2},1]$, which is a subset of the computational domain $[0,1]$. The results are reported in Table \ref{tab:advection-propagation}. Note that the global errors of both tables are almost identical. The local errors, however, differ significantly. The results of Table~\ref{tab:advection-propagation} show that even the uncorrected method is locally third order accurate in $[\tfrac{1}{2},1]$. This shows that order reduction of the local error only happens at the inflow boundary. \black The second order error made at this boundary is then propagated by a locally third order accurate scheme and thus no further decrease in the order is observed for the global error.

\begin{table}[htb]
	\begin{center}
	\textbf{Local error}
	\vspace{0.1cm}

	\begin{tabular}{rrrrrrrrr}
		& \multicolumn{2}{c}{unmodified} && \multicolumn{2}{c}{TDBC} && \multicolumn{2}{c}{CEC} \\
		\cline{2-3} \cline{5-6} \cline{8-9}
		\multicolumn{1}{l}{step size} & \multicolumn{1}{l}{$l^{\infty}$ error} & \multicolumn{1}{l}{order} & &
		\multicolumn{1}{l}{$l^{\infty}$ error} & \multicolumn{1}{l}{order} & & \multicolumn{1}{l}{$l^{\infty}$ error} & \multicolumn{1}{l}{order} \\
		\hline
		2.40e-01    & 1.44e-02    & --     && 1.44e-02    & --    && 7.19e-03    & --      \\
		1.20e-01    & 2.05e-03    & 2.81   && 2.05e-03    & 2.81  && 9.33e-04    & 2.95    \\
		6.00e-02    & 2.75e-04    & 2.90   && 2.75e-04    & 2.90  && 1.24e-04    & 2.91    \\
		3.00e-02    & 3.58e-05    & 2.95   && 3.58e-05    & 2.95  && 1.62e-05    & 2.94    \\
		1.50e-02    & 4.56e-06    & 2.97   && 4.56e-06    & 2.97  && 2.08e-06    & 2.96    \\
		7.50e-03    & 5.76e-07    & 2.99   && 5.76e-07    & 2.99  && 2.63e-07    & 2.98    \\
	\end{tabular}
	\vspace{0.25cm}

	\textbf{Global error}
	\vspace{0.1cm}

	\begin{tabular}{rrrrrrrrr}
		& \multicolumn{2}{c}{unmodified} && \multicolumn{2}{c}{TDBC} && \multicolumn{2}{c}{CEC} \\
		\cline{2-3} \cline{5-6} \cline{8-9}
		\multicolumn{1}{l}{step size} & \multicolumn{1}{l}{$l^{\infty}$ error} & \multicolumn{1}{l}{order} & &
		\multicolumn{1}{l}{$l^{\infty}$ error} & \multicolumn{1}{l}{order} && \multicolumn{1}{l}{$l^{\infty}$ error} & \multicolumn{1}{l}{order} \\
		\hline
		2.40e-01    & 3.73e-01    & --    && 1.02e-01    & --    && 2.59e-02    & --     \\
		1.20e-01    & 9.07e-02    & 2.04  && 3.31e-02    & 1.63  && 4.72e-03    & 2.46   \\
		6.00e-02    & 1.26e-02    & 2.85  && 9.46e-03    & 1.81  && 9.71e-04    & 2.28   \\
		3.00e-02    & 1.87e-03    & 2.75  && 2.52e-03    & 1.91  && 3.16e-04    & 1.62   \\
		1.50e-02    & 4.88e-04    & 1.94  && 6.51e-04    & 1.95  && 8.99e-05    & 1.81   \\
		7.50e-03    & 1.25e-04    & 1.97  && 1.65e-04    & 1.98  && 2.39e-05    & 1.91   \\
	\end{tabular}
	\end{center}

	\caption{The local (at $t=0$) and global errors computed in $[\tfrac{1}{2},1]$ for the unmodified Strang splitting as well as the second order TDBC and CEC corrected Strang splitting applied to equation (\ref{eq:conservation-law}) with $f(u)=\e^{u-1}$ are shown. The space discretization is conducted by using a second order upwind finite difference stencil with $500$ grid points. All problems are integrated until $t=1.9$ and use the initial value $u(0,x)=1+x$.
	\label{tab:advection-propagation}}
\end{table}

To conclude this section we perform a more thorough comparison of the accuracy that is achieved by the CEC and the TDBC correction. To that end we show the relative advantage (in accuracy) of the CEC approach in Table~\ref{tab:comparison} for five different reaction terms and five different advection coefficients. In almost all cases the CEC correction is more accurate compared to the TDBC correction. Depending on the problem the increase in accuracy can be more than an order of magnitude. We also observe that for both methods employing the third order correction can yield significant gains for some problems, while significantly diminishing the accuracy for other problems.

\begin{table}[htb]
	\begin{center}
	$t=0.5$
	\vspace{0.1cm}

	\begin{tabular}{r|rrrrr}
		& $a_1$ & $a_2$ & $a_3$ & $a_4$ & $a_5$ \\
		\hline
		$f_1$  & 27.6(6.9,0.8) & 4.5(1.2,0.9) & 27.3(23,1.5) & 14.1(7.2,2.2) & 4.0(1.0,0.7) \\
		$f_2$  & 18.2(1.5,0.9) & 15.4(1.3,1.0) & 14.1(7.2,2.2) & 9.4(0.2,1.0) & 8.2(0.9,0.7) \\
		$f_3$  & 22.7(5.7,0.8) & 4.6(1.3,0.9) & 15.8(13.7,1.5) & 4.0(0.3,1.0) & 4.2(1.1,0.7) \\
		$f_4$  & 5.7(3.1,0.6) & 2.2(1.4,0.7) & 1.5(3.3,0.6) & 1.7(0.4,1.0) & 2.7(1.4,0.6) \\
		$f_5$  & 2.4(1.0,0.7) & 2.4(1.0,0.9) & 2.5(1.0,1.7) & 3.7(1.0,1.0) & 3.4(1.0,0.7)
	\end{tabular}
	\end{center}
	
	\begin{center}
	$t=2$
	\vspace{0.1cm}

	\begin{tabular}{r|rrrrr}
		& $a_1$ & $a_2$ & $a_3$ & $a_4$ & $a_5$ \\
		\hline
		$f_1$  & 21.6(2.9,0.6) & 5.7(0.8,0.4) & 35.3(24.7,1.3) & 2.6(0.8,0.5) & 3.9(1.9,0.4) \\
		$f_2$  & 11.1(1.4,0.5) & 19.3(3.8,0.3) & 15.5(7.1,2.0) & 0.9(0.9,0.2) & 6.0(1.6,0.5) \\
		$f_3$  & 18.6(2.6,0.6) & 5.7(0.8,0.4) & 19.3(13.9,1.3) & 2.5(1.1,0.4) & 3.8(1.9,0.4) \\
		$f_4$  & 6.8(1.4,0.6) & 8.8(1.8,0.5) & 1.2(29,0.4) & 4.0(2.7,0.2) & 2.1(1.5,0.4) \\
		$f_5$  & 1.7(1.0,0.4) & 1.0(1.0,0.2) & 2.6(1.0,1.6) & 0.8(1.0,0.5) & 2.1(1.0,0.5)
	\end{tabular}
	\end{center}
	\caption{The accuracy (at time $t=0.5$ and $t=2$) of the best TDBC approach (this can be the second or third order correction) divided by the accuracy of the best CEC approach is shown for five different reactions $f_1=\sqrt{u+1}$, $f_2=\e^{u/5}$, $f_3=\log(2+u)$, $f_4=1/2+\text{arsinh\hskip 1pt}{u}$, $f_5=\cos u$ and five different advection coefficients $a_1=1+\sin x$, $a_2=\sin(\pi x/2)+2/5$, $a_3=3/2-x$, $a_4=1/5+\e^{-50 (x-1/2)^2}$, $a_5=1 + \sin(2\pi x)/5$. The number in parentheses shows the gain in accuracy achieved by going from CEC2 to CEC3 and from TDBC2 to TDBC3, respectively (values larger than one indicate a gain in accuracy, while values smaller than one indicate a loss in accuracy). The space discretization is conducted by using a second order upwind finite difference stencil with $500$ grid points. \label{tab:comparison}}
\end{table}

\subsection{Dispersive problem}

We consider the following dispersive equation
\begin{equation} \partial_t u(t,x) = i \partial_{xx} u(t,x) + f(u(t,x)), \qquad u(t,0)=u(t,1)=1, \qquad u(0,x)=u_0(x) \label{eq:dispersion} \end{equation}
	and once again compare the unmodified Strang splitting with both the TDBC correction and the CEC correction. The corresponding numerical results are shown in Table \ref{tab:dispersion-2nd}. The local error agrees very well with the convergence analysis conducted in section \ref{sec:convergence} and \cite{einkemmer2015}. On the other hand, the behavior of the global error is rather erratic.
We should emphasize, however, that this is certainly not in contradiction to our convergence analysis. In any case, we observe that the accuracy of both corrections is clearly superior to the unmodified Strang splitting. We also note that, for this example, both corrections perform almost identical.

\begin{table}[htb]
	\begin{center}
	\textbf{Local error}
	\vspace{0.1cm}

	\begin{tabular}{rrrrrrrrr}
		& \multicolumn{2}{c}{unmodified} && \multicolumn{2}{c}{TDBC} && \multicolumn{2}{c}{CEC} \\
		\cline{2-3} \cline{5-6} \cline{8-9}
		\multicolumn{1}{l}{step size} & \multicolumn{1}{l}{$l^{\infty}$ error} & \multicolumn{1}{l}{order} & &
		\multicolumn{1}{l}{$l^{\infty}$ error} & \multicolumn{1}{l}{order} & & \multicolumn{1}{l}{$l^{\infty}$ error} & \multicolumn{1}{l}{order} \\
		\hline
		1.20e-02    & 5.84e-03    & --    && 1.48e-03    & --     && 1.50e-03    & --    \\
		6.00e-03    & 2.79e-03    & 1.07  && 2.72e-04    & 2.45   && 2.70e-04    & 2.47  \\
		3.00e-03    & 1.23e-03    & 1.19  && 3.49e-05    & 2.96   && 3.47e-05    & 2.96  \\
		1.50e-03    & 6.38e-04    & 0.94 && 8.77e-06    & 1.99   && 8.65e-06    & 2.00   \\
		7.50e-04    & 2.95e-04    & 1.11  && 2.11e-06    & 2.05   && 2.08e-06    & 2.05  \\
		3.75e-04    & 1.30e-04    & 1.18  && 5.15e-07    & 2.04   && 5.07e-07    & 2.04  \\
	\end{tabular}
	\vspace{0.25cm}

	\textbf{Global error}
	\vspace{0.1cm}

	\begin{tabular}{rrrrrrrrr}
		& \multicolumn{2}{c}{unmodified} && \multicolumn{2}{c}{TDBC} && \multicolumn{2}{c}{CEC} \\
		\cline{2-3} \cline{5-6} \cline{8-9}
		\multicolumn{1}{l}{step size} & \multicolumn{1}{l}{$l^{\infty}$ error} & \multicolumn{1}{l}{order} & &
		\multicolumn{1}{l}{$l^{\infty}$ error} & \multicolumn{1}{l}{order} && \multicolumn{1}{l}{$l^{\infty}$ error} & \multicolumn{1}{l}{order} \\
		\hline
		1.20e-02    & 2.02e-02    & --    && 2.35e-03    & 1.89  && 2.33e-03    & --    \\
		6.00e-03    & 1.18e-02    & 0.78   && 5.32e-04    & 2.14  && 5.25e-04    & 2.15    \\
		3.00e-03    & 4.77e-03    & 1.30    && 1.09e-04    & 2.28  && 1.08e-04    & 2.29    \\
		1.50e-03    & 1.04e-03    & 2.20     && 4.85e-05    & 1.17  && 4.82e-05    & 1.16    \\
		7.50e-04    & 6.09e-04    & 0.77    && 1.82e-05    & 1.41  && 1.82e-05    & 1.41    \\
		3.75e-04    & 2.20e-04    & 1.47    && 1.44e-05    & 0.34 && 1.44e-05    & 0.34   \\
		1.88e-04    & 9.37e-05    & 1.23    && 4.70e-07    & 4.94  && 4.66e-07    & 4.95    \\
	\end{tabular}
	\end{center}

	\caption{The local (at $t=0$) and global errors for the unmodified Strang splitting as well as the second order TDBC and CEC corrected Strang splitting applied to equation (\ref{eq:dispersion}) with $f(u)=\e^{u-1}$ are shown. The space discretization is conducted by using the standard centered finite difference stencil with $200$ grid points. All problems are integrated until $t=0.19$ and use the initial value $u(0,x)=1+\sin \pi x + i \sin 2 \pi x$.
	\label{tab:dispersion-2nd}}
\end{table}

Now, let us compare these results with the third order corrections shown in Table \ref{tab:dispersion-3rd}. Note that in this case it is once again necessary to compute the first derivative by numeric differentiation. For both corrected versions the error is worse by approximately a factor of three compared to the second order corrections. It should be noted, however, that the convergence is much more predictable for the third order correction. That is, we do not observe the erratic convergence behavior described above. This might be of some interest in practice as an automatic step size controller would assume such a regular behavior (if this is not the case multiple and frequent step size rejection might occur).

\begin{table}[htb]
	\begin{center}
	\textbf{Local error}
	\vspace{0.1cm}

	\begin{tabular}{rrrrrrrrr}
		& \multicolumn{2}{c}{unmodified} && \multicolumn{2}{c}{TDBC} && \multicolumn{2}{c}{CEC} \\
		\cline{2-3} \cline{5-6} \cline{8-9}
		\multicolumn{1}{l}{step size} & \multicolumn{1}{l}{$l^{\infty}$ error} & \multicolumn{1}{l}{order} & &
		\multicolumn{1}{l}{$l^{\infty}$ error} & \multicolumn{1}{l}{order} & & \multicolumn{1}{l}{$l^{\infty}$ error} & \multicolumn{1}{l}{order} \\
		\hline
		1.20e-02    & 5.84e-03    & --    && 1.51e-03    & --     && 1.53e-03    & --      \\
		6.00e-03    & 2.79e-03    & 1.07  && 2.48e-04    & 2.61   && 2.38e-04    & 2.69    \\
		3.00e-03    & 1.23e-03    & 1.19  && 2.92e-05    & 3.09   && 2.81e-05    & 3.08    \\
		1.50e-03    & 6.38e-04    & 0.94 && 3.21e-06    & 3.18   && 3.14e-06    & 3.16    \\
		7.50e-04    & 2.95e-04    & 1.11  && 3.82e-07    & 3.07   && 3.72e-07    & 3.08    \\
		3.75e-04    & 1.30e-04    & 1.18  && 4.65e-08    & 3.04   && 4.64e-08    & 3.00     \\
	\end{tabular}
	\vspace{0.25cm}

	\textbf{Global error}
	\vspace{0.1cm}

	\begin{tabular}{rrrrrrrrr}
		& \multicolumn{2}{c}{unmodified} && \multicolumn{2}{c}{TDBC}  && \multicolumn{2}{c}{CEC} \\
		\cline{2-3} \cline{5-6} \cline{8-9}
		\multicolumn{1}{l}{step size} & \multicolumn{1}{l}{$l^{\infty}$ error} & \multicolumn{1}{l}{order} & &
		\multicolumn{1}{l}{$l^{\infty}$ error} & \multicolumn{1}{l}{order} && \multicolumn{1}{l}{$l^{\infty}$ error} & \multicolumn{1}{l}{order} \\
		\hline
		1.20e-02    & 2.02e-02    & --   && 9.02e-03    & 2.09   && 7.91e-03    & --  \\
		6.00e-03    & 1.18e-02    & 0.78  && 1.84e-03    & 2.29   && 1.73e-03    & 2.19  \\
		3.00e-03    & 4.77e-03    & 1.30    && 4.16e-04    & 2.15   && 4.12e-04    & 2.07  \\
		1.50e-03    & 1.04e-03    & 2.20   && 9.85e-05    & 2.08   && 9.97e-05    & 2.05  \\
		7.50e-04    & 6.09e-04    & 0.77   && 2.42e-05    & 2.03   && 2.44e-05    & 2.03  \\
		3.75e-04    & 2.20e-04    & 1.47   && 6.01e-06    & 2.01   && 6.01e-06    & 2.02   \\
		1.88e-04    & 9.37e-05    & 1.23   && 1.45e-06    & 2.06   && 1.50e-06    & 2.00    \\
	\end{tabular}
	\end{center}

	\caption{The local (at $t=0$) and global errors for the unmodified Strang splitting as well as the third order TDBC and CEC corrected Strang splitting applied to equation (\ref{eq:dispersion}) with $f(u)=\e^{u-1}$ are shown. The space discretization is conducted by using the standard centered finite difference stencil with $200$ grid points. All problems are integrated until $t=0.19$ and use the initial value $u(0,x)=1+\sin \pi x + i \sin 2 \pi x$.
	\label{tab:dispersion-3rd}}
\end{table}

To conclude this section let us investigate the erratic convergence behavior for Strang splitting and the two second order corrections. Note that since the parabolic smoothing property does not apply in this case we can lose up to an order by going from the local to the global error. However, whether this actually happens depends on the precise step size chosen (a phenomenon called resonance; see, for example \cite{hochbruck1999,grimm2006}). Now, for the third order corrections this is not an issue as only order two can be attained globally in any case. Thus, the global error for this scheme behaves as we would expect from a second order method. The behavior described is in complete agreement with Figure \ref{fig:res-step}. In addition, it is interesting to look at the behavior of the local and the global error as a function of time for various step sizes. The corresponding results are shown in Figure \ref{fig:res-evol} and illustrate these resonances from the perspective of global error propagation.

\begin{figure}[htb]
	\begin{center}
		\includegraphics[width=10cm]{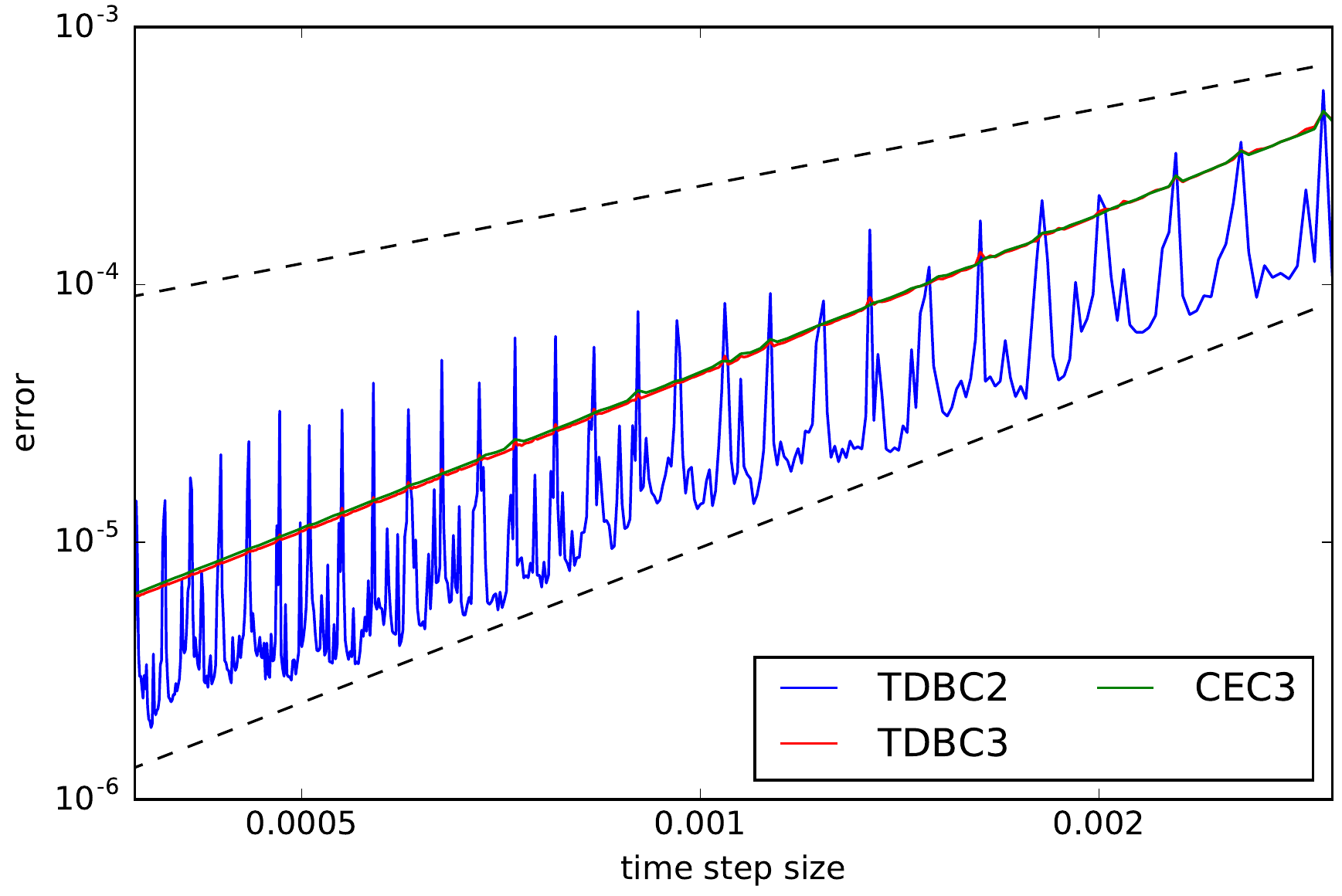}
	\end{center}
	\caption{The global error in the infinity norm as a function of the time step size is shown. The error for TDBC2 and CEC2 is almost identical and therefore only the (erratic) error for TDBC2 is shown in the plot. In addition, the dashed lines are of slope $1$ and $2$, respectively.
	In all simulations equation (\ref{eq:dispersion}) with $f(u)=\e^{u-1}$ is employed and the initial value $u(0,x)=1+\sin \pi x + i \sin 2 \pi x$ is imposed.	
	The space discretization is conducted by using the standard centered finite difference stencil with $200$ grid points and all simulations are conducted until $t=0.19$.\label{fig:res-step}
}
\end{figure}

\begin{figure}[htb]
	\begin{center}
		\includegraphics[width=10cm]{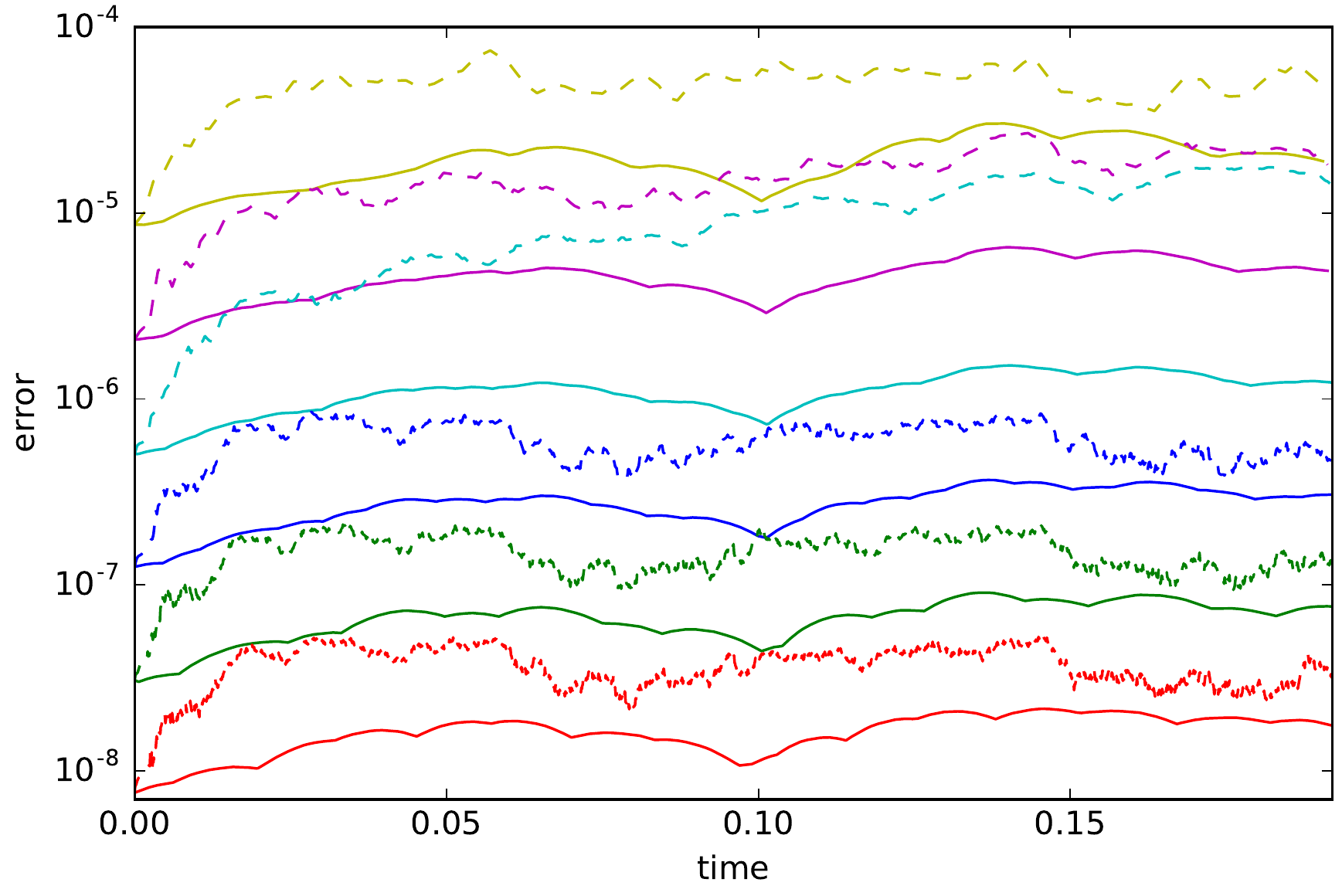}

		\includegraphics[width=10cm]{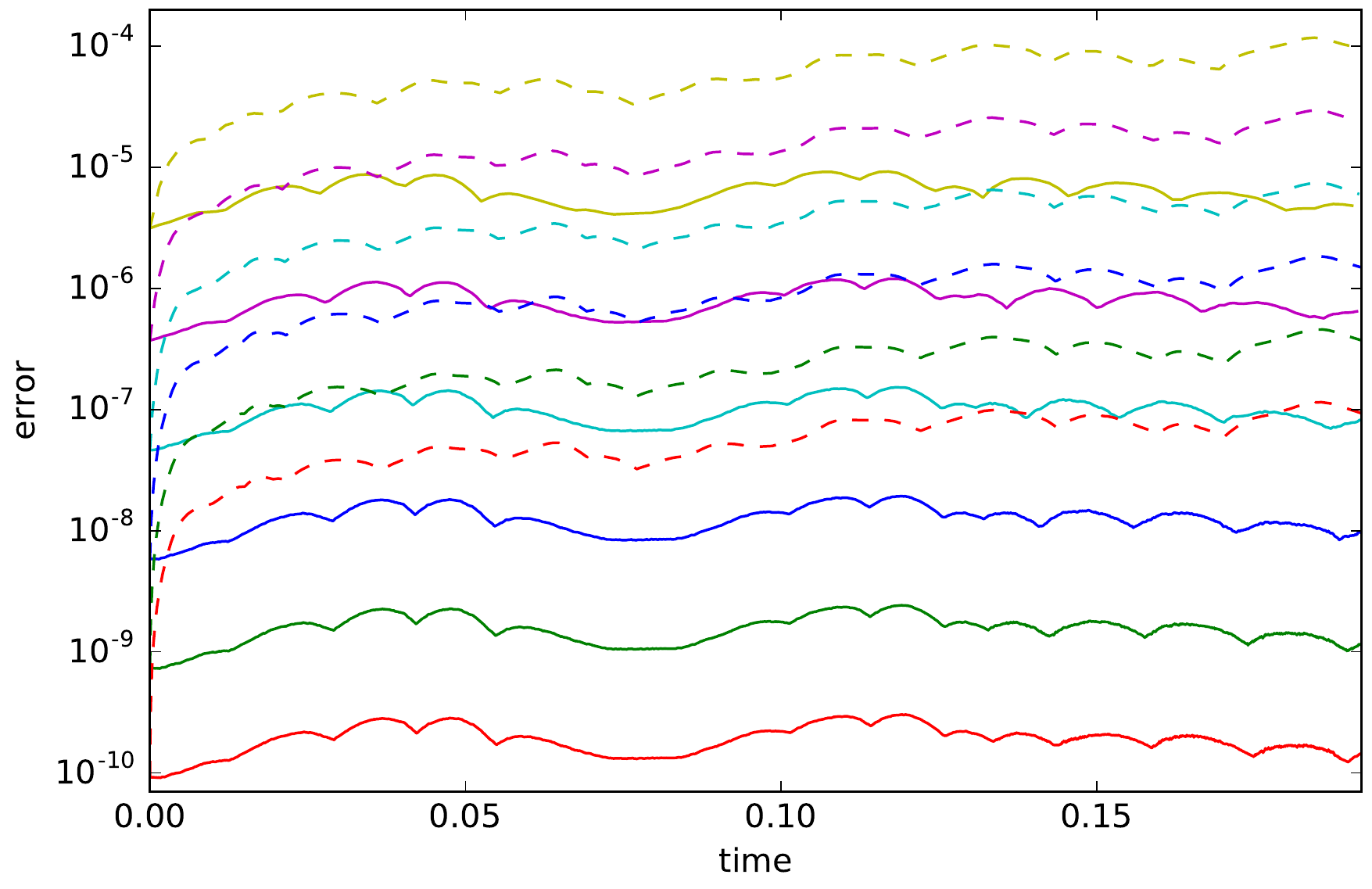}
	\end{center}
	\caption{The local (full lines) and global errors (dashed lines) in the infinity norm are shown as a function of time for the second order CEC (top) and the third order CEC (bottom) corrections. The following step sizes are used (from top to bottom in this order in both cases): $1.5\cdot10^{-3}$ (yellow), $7.5\cdot10^{-4}$ (magenta), $3.75\cdot10^{-4}$ (cyan), $1.88\cdot10^{-4}$ (blue), $9.38\cdot10^{-5}$ (green), $4.69\cdot10^{-5}$ (red). In all simulations equation (\ref{eq:dispersion}) with $f(u)=\e^{u-1}$ is employed and the initial value $u(0,x)=1+\sin \pi x + i \sin 2 \pi x$ is imposed. The space discretization is conducted by using the standard centered finite difference stencil with $200$ grid points.\label{fig:res-evol}}
\end{figure}

\section{Conclusion \label{sec:conclusion}}

In this paper we have performed a mathematically rigorous convergence analysis of the TDBC method for a non-linear problem, thus complementing the results that have been obtained earlier for the CEC method \cite{einkemmer2015}. This result agrees very well with the numerical simulations conducted.

Both methods have certain advantages and disadvantages from an implementation point of view. Furthermore, we have found that in most cases the accuracy of both methods for the second order PDEs considered is quite similar. However, for the advection-reaction problem the CEC method significantly outperforms the TDBC method (for some problems by more than an order of magnitude).

We also find that whether performing further corrections to obtain a numerical scheme of locally third order accuracy is advantageous depends on the specific problem considered. Gains by more than an order of magnitude as well as a diminishing of the accuracy by at least a factor of $5$ have been observed. In addition, the third order correction results in a predictable convergence behavior for the dispersion-reaction equation (which is not the case if only the second order correction is applied).

\bibliographystyle{plain}
\bibliography{papers}

\begin{thebibliography}{10}

\bibitem{alonsomallo2016}
I.~Alonso-Mallo, B.~Cano, and N.~Reguera.
\newblock {Avoiding order reduction when integrating linear initial boundary
  value problems with exponential splitting methods}.
\newblock {\em Private copy}, 2016.

\bibitem{alonsomallo2015}
I.~Alonso-Mallo, B.~Cano, and N.~Reguera.
\newblock {Avoiding order reduction when integrating linear initial boundary
  value problems with Lawson methods}.
\newblock {\em IMA J. Numer. Anal.}, in press.

\bibitem{bao2002}
W.~Bao, S.~Jin, and P.A. Markowich.
\newblock {On time-splitting spectral approximations for the Schr{\"o}dinger
  equation in the semiclassical regime}.
\newblock {\em J. Comput. Phys.}, 175(2):487--524, 2002.

\bibitem{canoreguera2017}
B.~Cano and N.~Reguera.
\newblock {Avoiding order reduction when integrating nonlinear Schr\"odinger
  equation with Strang method}.
\newblock {\em J. Comput. Appl. Math.}, 316:86--99, 2017.

\bibitem{carpenter1993}
M.H. Carpenter, D.~Gottlieb, S.~Abarbanel, and W.S. Don.
\newblock {The theoretical accuracy of Runge--Kutta time discretizations for
  the initial boundary value problem: a study of the boundary error}.
\newblock {\em SIAM J. Sci. Comput.}, 16:1241--1252, 1995.

\bibitem{casas2015}
F.~Casas, N.~Crouseilles, E.~Faou, and M.~Mehrenberger.
\newblock {High-order Hamiltonian splitting for Vlasov--Poisson equations}.
\newblock {\em Numer. Math.}, 135(3):769--801, 2017.

\bibitem{cheng1976}
C.~Cheng and G.~Knorr.
\newblock {The integration of the Vlasov equation in configuration space}.
\newblock {\em J. Comput. Phys.}, 22(3):330--351, 1976.

\bibitem{connors2014}
J.M. Connors, J.W. Banks, J.A. Hittinger, and C.S. Woodward.
\newblock {Quantification of errors for operator-split advection--diffusion
  calculations}.
\newblock {\em Comput. Methods Appl. Mech. Engrg.}, 272:181--197, 2014.

\bibitem{einkemmer1401}
N.~Crouseilles, L.~Einkemmer, and E.~Faou.
\newblock {A Hamiltonian splitting for the Vlasov--Maxwell system}.
\newblock {\em J. Comput. Phys.}, 238:224--240, 2015.

\bibitem{einkemmer1602}
N.~Crouseilles, L.~Einkemmer, and E.~Faou.
\newblock {An asymptotic preserving scheme for the relativistic Vlasov--Maxwell
  equations in the classical limit}.
\newblock {\em Comput. Phys. Commun.}, 209:13--26, 2016.

\bibitem{dawson1991}
C.N. Dawson and M.F. Wheeler.
\newblock Time-splitting methods for advection-diffusion-reaction equations
  arising in contaminant transport.
\newblock In R.~O'Malley, editor, {\em Proceedings of ICIAM 91 (Washington, DC,
  1991)}, pages 71--82. SIAM, Philadelphia, 1992.

\bibitem{descombes2001}
S.~Descombes.
\newblock Convergence of a splitting method of high order for
  reaction-diffusion systems.
\newblock {\em Math. Comp.}, 70:1481--1501, 2001.

\bibitem{einkemmer1207}
L.~Einkemmer and A.~Ostermann.
\newblock {Convergence analysis of Strang splitting for Vlasov-type equations}.
\newblock {\em SIAM J. Numer. Anal.}, 52(1):140--155, 2014.

\bibitem{einkemmer1408}
L.~Einkemmer and A.~Ostermann.
\newblock {A splitting approach for the Kadomtsev--Petviashvili equation}.
\newblock {\em J. Comput. Phys.}, 299:716--730, 2015.

\bibitem{einkemmer2015}
L.~Einkemmer and A.~Ostermann.
\newblock Overcoming order reduction in diffusion-reaction splitting. {P}art 1:
  {D}irichlet boundary conditions.
\newblock {\em SIAM J. Sci. Comput.}, 37(3):A1577--A1592, 2015.

\bibitem{einkemmer2016}
L.~Einkemmer and A.~Ostermann.
\newblock {Overcoming order reduction in diffusion-reaction splitting. Part 2:
  oblique boundary conditions}.
\newblock {\em SIAM J. Sci. Comput.}, 38(6):A3741--A3757, 2016.

\bibitem{faou2012}
E.~Faou.
\newblock {\em Geometric numerical integration and Schr\"odinger equations}.
\newblock European Mathematical Society, Z\"urich, 2012.

\bibitem{gerisch2002}
A.~Gerisch and J.G. Verwer.
\newblock {Operator splitting and approximate factorization for
  taxis-diffusion-reaction models}.
\newblock {\em Appl. Numer. Math.}, 42(1):159--176, 2002.

\bibitem{grandgirard2006}
V.~Grandgirard, M.~Brunetti, P.~Bertrand, N.~Besse, X.~Garbet, P.~Ghendrih,
  G.~Manfredi, Y.~Sarazin, O.~Sauter, E.~Sonnendr{\"u}cker, J.~Vaclavik, and
  L.~Villard.
\newblock {A drift-kinetic Semi-Lagrangian 4D code for ion turbulence
  simulation}.
\newblock {\em J. Comput. Phys.}, 217:395--423, 2006.

\bibitem{grimm2006}
V.~Grimm and M.~Hochbruck.
\newblock {Error analysis of exponential integrators for oscillatory
  second-order differential equations}.
\newblock {\em J. Phys. A: Math. Gen.}, 39(19):5495--5507, 2006.

\bibitem{hochbruck1999}
M.~Hochbruck and C.~Lubich.
\newblock {Exponential integrators for quantum-classical molecular dynamics}.
\newblock {\em BIT}, 39(4):620--645, 1999.

\bibitem{holden2013}
H.~Holden, C.~Lubich, and N.~Risebro.
\newblock {Operator splitting for partial differential equations with Burgers
  nonlinearity}.
\newblock {\em Math. Comp.}, 82(281):173--185, 2013.

\bibitem{hundsdorfer1995}
W.~Hundsdorfer and J.G. Verwer.
\newblock {A note on splitting errors for advection-reaction equations}.
\newblock {\em Appl. Numer. Math.}, 18(1):191--199, 1995.

\bibitem{hundsdorfer2003book}
W.~Hundsdorfer and J.G. Verwer.
\newblock {\em {Numerical solution of time-dependent
  advection-diffusion-reaction equations}}.
\newblock Springer, Berlin, 2003.

\bibitem{klein2011}
C.~Klein and K.~Roidot.
\newblock {Fourth order time-stepping for Kadomtsev--Petviashvili and
  Davey--Stewartson equations}.
\newblock {\em SIAM J. Sci. Comput.}, 33(6):3333--3356, 2011.

\bibitem{leveque1983}
R.J. LeVeque and J.~Oliger.
\newblock {Numerical methods based on additive splittings for hyperbolic
  partial differential equations}.
\newblock {\em Math. Comp.}, 40(162):469--497, 1983.

\bibitem{lubich2008}
C.~Lubich.
\newblock {On splitting methods for Schr{\"o}dinger--Poisson and cubic
  nonlinear Schr{\"o}dinger equations}.
\newblock {\em Math. Comp.}, 77(264):2141--2153, 2008.

\bibitem{spee1998}
E.J. Spee, J.G. Verwer, P.M. de~Zeeuw, J.G. Blom, and W.~Hundsdorfer.
\newblock {A numerical study for global atmospheric transport-chemistry
  problems}.
\newblock {\em Math. Comput. Simulat.}, 48(2):177--204, 1998.

\end{thebibliography}

\end{document}